\documentclass{article}
\usepackage{amsmath, amsthm, amsfonts}
\usepackage{mathrsfs}
\usepackage[linesnumbered,ruled,vlined]{algorithm2e}
\usepackage{xcolor}
\usepackage{multirow}
\usepackage{caption}
\usepackage{float}
\usepackage{subcaption}
\usepackage{graphicx}
\usepackage{authblk}

 \newtheorem{thm}{Theorem}[section]
 
 \newtheorem{lem}[thm]{Lemma}
 \theoremstyle{definition}
 \newtheorem{defn}[thm]{Definition}
  \newtheorem{prop}[thm]{Proposition}
 \theoremstyle{remark}

 \numberwithin{equation}{section}
\DeclareMathOperator*\dif{\mathop{}\!\mathrm{d}}

\newcommand{\E}{\mathbb{E}_i}
\newcommand{\dd}{\,{\dif}}
\newcommand{\ds}{\,{\dif}s}

\newcommand{\dr}{\,{\dif}r}

\newcommand{\dwr}{\,{\dif}W_r}
\newcommand{\dbt}{\,{\dif}B_t}
\newcommand{\dbs}{\,{\dif}B_s}
\newcommand{\dbr}{\,{\dif}{B}_r}

\SetCommentSty{mycommfont}

\SetKwInput{KwInput}{Input}                % Set the Input
\SetKwInput{KwOutput}{Output}              % set the Output

\begin{document}

\title{A predictor-corrector deep learning algorithm for high dimensional stochastic partial differential equations}

\author[1]{He Zhang \thanks{\texttt{Email: he$\_$zhang@jlu.edu.cn.}~The work of this author was supported in part by NSFC Grant 12201242.}}
%The work of this author was supported in part by NSFC Grant by  NSFC (Grant Nos. 12201242, 12071175) and the National Research Program of China
%(Grant No. 2013CB834100), by the Natural Science Foundation of Jilin Province
%(Grant Nos. 20200201253JC, 201902013020JC), and by the Project of Science and
%Technology Development of Jilin Province, China (Grant No. 2017C028-1).
\author[1]{Ran Zhang\thanks{\texttt{Email: zhangran@jlu.edu.cn.}~Corresponding author.~The work of this author was supported in part by NSFC Grant 11971198}}
\author[2]{Tao Zhou\thanks{\texttt{Email: tzhou@lsec.cc.ac.cn.}~ The work of this author was supported in part by NSFC Grants 11731006, 11831010 and 11871068.}}
\affil[1]{
School of Mathematics, Jilin University, Changchun, Jilin 130012, China}
\affil[2]{Institute of Computational Mathematics and Scientific/Engineering Computing, Academy of Mathematics
and Systems Science, Chinese Academy of Sciences, Beijing, 100190, China}

\date{\today}

\maketitle

\begin{abstract}
In this paper, we present a deep learning-based numerical method for approximating high dimensional stochastic partial differential equations (SPDEs). At each time step, our method relies on a predictor-corrector procedure. More precisely, we decompose the original SPDE into a degenerate SPDE and a deterministic PDE. Then in the prediction step, we solve the degenerate SPDE with the Euler scheme, while in the correction step we solve the second-order deterministic PDE by deep neural networks via its equivalent backward stochastic differential equation (BSDE). Under standard assumptions, error estimates and the rate of convergence of the proposed algorithm are presented. The efficiency and accuracy of the proposed algorithm are illustrated by numerical examples.
\\
\end{abstract}


\section{Introduction}
This paper is concerned with numerical methods for the following stochastic partial differential equation (SPDE):
\begin{equation}\label{spde}
  \begin{split}
  u(t,x)=&h(x)+\int_t^T \left(\mathcal{L}u(s,x)+f(s,x,u(s,x),(\nabla u\sigma)(s,x))\right)\ds\\
  &+\int_t^Tg(s,x,u(s,x))\dbs,
  \end{split}
\end{equation}
where $u\colon [0,T]\times\mathbb{R}^d\rightarrow\mathbb{R}^k$, $\mathcal{L}u=[Lu_1,\cdots,Lu_k]^\top$ with
\begin{equation*}
L=\frac{1}{2}\sum_{i,j=1}^d(\sigma\sigma^*)_{ij}(t,x)\frac{\partial^2}{\partial x_i \partial x_j}+\sum_{i=1}^d \mu_i(t,x)\frac{\partial}{\partial x_i}.
\end{equation*}
Here $\nabla$ denotes the space gradient, $B$ is a Brownian motion defined on some complete probability space $(\Omega,\mathcal{F},P)$, and the stochastic integral ${\dif}B$ is a backward It\^{o} integral. SPDEs of this type appear in many real-world applications such as filtering problems, population genetics, and statistical hydromechanics (see e.g. \cite{MR2391779} and the references therein).

Numerical algorithms for the above SPDEs have been studied extensively in recent years. A comprehensive review of the literature can be found in Jentzen and Kloeden \cite{MR2578878}. For low dimensional problems, traditional space-time concretization strategies can be adopted to obtain reasonable numerical accuracy, and this includes the finite difference methods \cite{MR2729440, MR1480861, MR1654030}, finite element methods \cite{MR928351,MR2136207}, spectral methods \cite{MR1402994}, and Wiener chaos decomposition \cite{MR2694901}, just to name a few. However, for higher dimensional problems, all these numerical approaches are infeasible due to their prohibitive computational costs and storage requirements (i.e., the so called curse of dimensionality). Consequently, probabilistic methods have been developed in recent years to mitigate the curse of dimensionality. %In particular, Milstein and Tretyakov \cite{MR2521279} proposed a numerical scheme for linear SPDEs based on the method of characteristics and Monte Carlo techniques.
An important starting point for designing probabilistic methods is due to Pardoux and Peng \cite{MR1258986}, in which they presented a probabilistic representation of the solution of \eqref{spde} through the solution of the following forward backward doubly stochastic differential equation (FBDSDE)
\begin{align}
  X^{t,x}_s&=x+\int_t^s \mu(r,X^{t,x}_r)\dr+\int_t^s \sigma(r,X^{t,x}_r)\dwr,\label{sde1}\\
  Y^{t,x}_s&=h(X_T^{t,x})+\int_s^T f(r,X^{t,x}_r,Y^{t,x}_r,Z^{t,x}_r)\dr\nonumber\\
  &-\int_s^T Z^{t,x}_r\dwr+\int_s^T g(r,X^{t,x}_r,Y^{t,x}_r)\dd B_r,\label{bdsde}
\end{align}
where $\{X^{t,x}_s,~t\leq s\leq T\}$ denotes the forward process starting from $x$ at time $t$, $W$ is a $d$-dimensional standard Brownian motion independent of $B$ in \eqref{spde} and \eqref{bdsde}. More precisely, it was shown in \cite{MR1258986} that if $u$ is a classical solution of \eqref{spde}, then $\{(Y^{t,x}_s= u(s,X^{t,x}_s),~Z^{t,x}_s=(\nabla u\sigma)(s,X^{t,x}_s)),~t\leq s\leq T\}$ is the unique solution of \eqref{bdsde}. Based on this, various numerical methods for SPDEs using the associated FBDSDEs are proposed \cite{MR3538011,mr3483162,MR3356530,MR4092406}. In particular, when $g$ vanishes, the above FBDSDEs become the so called backward stochastic differential equations (BSDEs), for which extensive numerical algorithms have been investigated \cite{MR2056536,MR3674067,MR2152657,MR3449792,MR3454368,MR4135389,MR2023027,mr2552079}. Nevertheless, designing efficient algorithms for high dimensional problems is still very challenging.

More recently, deep learning-based numerical methods have attracted great attention for dealing with high dimensional problems. In particular, E et al \cite{MR3627592,MR3847747} presented a deep leaning method for high-dimensional parabolic PDEs by using the nonlinear Feyman-Kac formula \cite{MR1176785}. This approach was extended to fully nonlinear second-order PDEs in \cite{MR3993178}. Backward induction schemes for high dimensional PDEs and variational inequalities were developed in \cite{MR4081911}. A multilevel Picard iteration scheme was derived in \cite{MR3946468} by combining the Feynman-Kac and Bismut-Elworthy-Li formulae with a multilevel Picard approximation.

Although there has been extensive research on deep learning-based numerical approximations for PDEs \cite{MR4188517,FZZ_2022,MR4402262,MR3881695,MR3874585}, however, very little  attention has been paid to SPDEs. One can refer to \cite{beck2020deep} for numerical attempts on approximating solutions of initial value problems of SPDEs by neural networks. We also mention the work of Teng et al \cite{teng2021solving} on deep learning-based numerical scheme for FBDSDEs, where they define equivalent control problems by treating $Y_t$ and $Z_t$ in FBDSDE as controls, and use deep neural networks as approximation tools for these controls.

In this work, we shall present a predictor-corrector (or, an operator spliting) deep learning scheme for the weak solution of \eqref{spde}. Consider the following partition of $[0,T]:$  $$\pi^N\colon0=t_0<t_1<\cdots<t_N=T.$$  Our approach generates an approximated sequence $\{u_i(x)\}_{i=0}^N$ defined by
\begin{equation}\label{sequence}
u_N=h(x),\quad  u_i =\Phi_{\Delta t_i}\Psi_{t_i,t_{i+1}}u_{i+1}, \quad i=N-1,\cdots,0,
\end{equation}
where $\Delta t_i=t_{i+1}-t_i$, $\{\Phi_{t},~t\geq0\}$ and $\{\Psi_{st},~0\leq s\leq t\}$ denote the solution operators associated with the equations
\begin{equation}\label{sde2}
  {\dif}\psi+g(t,x,\psi(t,x))\dbt=0,~\lim\limits_{t\uparrow t_{i+1}}\psi_{t}=u_{i+1}=\varphi_{t_{i+1}},
\end{equation}
and
\begin{equation}\label{pde}
\partial_t\varphi+\mathcal{L}\varphi(t,x)+f(t,x,\varphi(t,x),(\nabla \varphi\sigma)(t,x))=0,~\lim\limits_{t\uparrow t_{i+1}}\varphi_{t}=\psi_{t_{i}},
\end{equation}
respectively.  Notice that at each time step, the SDE \eqref{sde2} serves as a predictor by providing an \emph{a priori} estimate of the solution at the next time step. Then we adopt its solution at $t_{i}$ as the terminal condition at $t_{i+1}$ of the PDE \eqref{pde}, and refine the initial estimate. Consequently, $u_i(x)=\varphi_{t_{i}}(x)$ is an approximation of the solution $u(t_i,X_{t_i}^{t,x}).$ We employ the Euler scheme for the SDE \eqref{sde2}, and the numerical approximation of the PDE \eqref{pde} is of primary interest. Inspired by \cite{MR3627592,MR3847747}, we shall propose to approximate \eqref{pde} by a deep neural network at each time step via its equivalent BSDE. Besides proposing the splitting scheme, our second contribution shall be the sharp convergence analysis of the proposed algorithm.

The main motivation for proposing the above splitting strategy is that it not only improves efficiency, but also provides remarkable flexibility in choosing different numerical methods for each sub-problem. Similar splitting ideas for linear SPDEs such as the Zakai equation have been discussed in \cite{MR1180784,MR1075210,MR2056536,MR1964941,MR1176783,MR1358091}.

%We also emphasize the connection between this approximation strategy and the uncertainty quantification problems in deep neural network \cite{dcds}. With properly chosen $f$ and $g$, the predictor \eqref{sde2} describes the dynamics of hidden units in continuous-depth residual networks with uncertainty, while the corrector \eqref{pde} is the corresponding adjoint equation.

The rest of the paper is organized as follows. In section 2, we provide with some preliminaries and present the discrete time approximation scheme for \eqref{spde} and its error estimates. In section 3, we present the deep learning algorithm based on the discrete time approximations in section 2. The convergence analysis of the proposed algorithm and the error estimates in terms of the approximation errors of the neural networks will also be presented. Numerical examples are presented in Section 4, and this is followed by some concluding remarks in Section 5.

\section{Preliminaries}

Throughout the paper, we denote by $(\Omega,\mathcal{F},P)$ a probability space on which two mutually independent standard Brownian motions $W$ and $B$ (with values in $\mathbb{R}^d$ and $\mathbb{R}^l$, respectively) are defined.  Let $\mathcal{F}_{s,t}^W$ and $\mathcal{F}_{s,t}^B$ be the
 the usual P-augmentation of $\sigma\{W_r-W_s;s\leq r\leq t\}$ and $\sigma\{B_r-B_s;s\leq r\leq t\}$, respectively. For each $t\in[0,T]$, we define two collections $\{\mathcal{F}_t\}_{0\leq t\leq T}$ and $\{\mathcal{G}_t\}_{0\leq t\leq T}$  by
\begin{equation*}
  \mathcal{F}_t:=\mathcal{F}_{0,t}^W\vee \mathcal{F}_{t,T}^B,~\text{and}~\mathcal{G}_{t}:=\mathcal{F}_{0,t}^W\vee \mathcal{F}_{0,T}^B.
\end{equation*}
Note that $\{\mathcal{F}_t\}_{0\leq t\leq T}$ is neither increasing nor decreasing, while $\{\mathcal{G}_t\}_{0\leq t\leq T}$ is an increasing filtration. We also denote by $|x|=(\sum_{i=1}^n|x_i|^2)^{1/2}$ the Euclidean norm of a vector $x\in\mathbb{R}^n$ and by $|A|=(\sum_{i}\sum_{j}|a_{ij}|^2)^{1/2}$ the Frobenius norm for a matrix $A\in\mathbb{R}^{m\times n}.$
Moreover, the following functional spaces will be frequently used:
\begin{itemize}
\item $C^k(\mathbb{R}^m;\mathbb{R}^n):$ the space of $C^k$-functions from $\mathbb{R}^m$ to $\mathbb{R}^n$;
\item $C^k_c(\mathbb{R}^m;\mathbb{R}^n):$ the space of compactly supported $C^k$-functions;
\item $C_{b}^k(\mathbb{R}^m;\mathbb{R}^n):$ the space of all $C^k$-functions whose partial derivatives up to order $k$ are bounded.
\item $L^p_{\mathcal{F}}(0,T;\mathbb{R}^n) (1\leq p<\infty):$ the space of all $\mathcal{F}_t$-measurable  $\mathbb{R}^n$-valued processes $\varphi(\cdot)$ such that
$$\mathbb{E}\left[\int_0^T |\varphi_t|^p\,{\dif}t\right]<\infty.$$
\item $L^p_{\mathcal{F}}(\Omega,C([0,T];\mathbb{R}^n)):$ the space of all $\mathcal{F}_t$-measurable $\mathbb{R}^n$-valued continuous processes $\varphi(\cdot)$ such that $$\mathbb{E}\left[\sup\limits_{0\leq t\leq T} |\varphi_t|^p\right]<\infty.$$
\item $L^2_{\omega}(\mathbb{R}^n):$  a weighted Hilbert space with scalar product $$(u,v)_{\omega}=\int_{\mathbb{R}^n}u(x)v(x)\omega(x)\dd x,$$ for which the induced the norm yields $|u|^2_{\omega}:=(u,u)_{\omega}$. Here $\omega(x)$ is a continuous positive function such that $\int_{\mathbb{R}^n}(1+|x|^2)\omega(x)\dd x<\infty.$
\item $H^k_{\omega}(\mathbb{R}^n) (k\geq0):$ the completion of $C_c^\infty(\mathbb{R}^n)$ with respect to the norm $|u|^2_{k,\omega}=\sum\limits_{0\leq |\alpha|\leq k}|D^{\alpha} u|^2_{\omega}$.

\item $\mathcal{H}$: the space of $\mathcal{F}_{t,T}^B$-measurable process $\varphi_t(\cdot)$ with values in $H^1_{\omega}(\mathbb{R}^n)$ such that $$|\varphi|^2_{\mathcal{H}}:=\mathbb{E}\left[\sup\limits_{0\leq t\leq T}|\varphi_t|^2_{\omega}\right]+\mathbb{E}\left[\int_0^T|\nabla\varphi\sigma|^2_{\omega}\dd t\right]<\infty.$$
\end{itemize}
In what follows, we also make the following assumptions.
\begin{itemize}
\item \textbf{(H1)}~The functions $f$ and $g$ are jointly measurable in $(t,x,y,z)$ such that $f(\cdot,\cdot,y,z)\in L^2_{\mathcal{F}}(0,T;\mathbb{R}^k)$ and $g(\cdot,\cdot,y)\in L^2_{\mathcal{F}}(0,T;\mathbb{R}^{k\times l}).$ Furthermore, we assume that there exists a constant $K>0$ such that
\begin{equation*}
   \begin{split}
       &\sup\limits_{0\leq t\leq T}\Big(|f(t,0,0,0)|^2+|g(t,0,0)|^2\Big)\leq K,\\
       &|f(t,x,y,z)-f(s,\bar{x},\bar{y},\bar{z})|^2+|g(t,x,y)-g(s,\bar{x},\bar{y})|^2\\
       &\quad\leq K\left(|t-s|+|x-\bar{x}|^2+|y-\bar{y}|^2+|z-\bar{z}|^2\right)
   \end{split}
\end{equation*}
for any  $t,s\in[0,T], ~x,\bar{x}\in\mathbb{R}^d,~y,\bar{y}\in\mathbb{R}^k,~z,\bar{z}\in\mathbb{R}^{k\times d}.$

\item \textbf{(H2)}~The terminal condition $h\in C_b^2(\mathbb{R}^{d};\mathbb{R}^k)$ satisfies the following Lipschitz condition
\begin{equation*}
  |h(x)-h(\bar{x})|^2\leq K|x-\bar{x}|^2,~\forall x,\bar{x}\in\mathbb{R}^d.
\end{equation*}

\item \textbf{(H3)}~The functions $\mu\in C_b^2(\mathbb{R}^{d};\mathbb{R}^d)$ and $\sigma\in C_b^2(\mathbb{R}^{d};\mathbb{R}^{d\times d})$ in \eqref{sde1} are uniformly $\frac{1}{2}$-H\"{o}lder continuous in $t\in[0,T]$ and uniformly Lipschitz continuous in $x\in\mathbb{R}^d$, i.e., there exists a $K>0$ such that
\begin{equation*}
  |\mu(t,x)-\mu(s,\bar{x})|^2+|\sigma(t,x)-\sigma(s,\bar{x})|^2\leq K\left(|t-s|+|x-\bar{x}|^2\right),~\forall x,\bar{x}\in\mathbb{R}^d.
\end{equation*}
Furthermore we assume that there exists a constant $c>0$, such that
\begin{equation*}
  \xi^T\sigma(t,x)\sigma^T(t,x)\xi\geq c|\xi|^2,~\forall~x,\xi\in\mathbb{R}^d,~t\in[0,T].
\end{equation*}
\end{itemize}
Under the above assumptions, it was shown in Pardoux and Peng \cite{MR1258986} that the BDSDE \eqref{bdsde} admits a unique solution
\begin{equation*}
  (Y,Z)\in L^2_{\mathcal{F}}\left(\Omega,C([0,T];\mathbb{R}^k)\right)\times L^2_{\mathcal{F}}\left(0,T;\mathbb{R}^{k\times d}\right).
\end{equation*}
We are now ready to introduce the following definition
\begin{defn}[\cite{MR1822898}]
A function $u\in\mathcal{H}$ is called a weak solution of \eqref{spde} if for every $\phi\in C_c^{\infty}([0,T]\times \mathbb{R}^d)$, it holds
\begin{equation*}
   \begin{split}
       & \quad \int_t^T\!\!\!\int_{\mathbb{R}^d}u(s,x)\partial_s\phi(s,x)\dd x\dd s-\int_t^T\!\!\!\int_{\mathbb{R}^d}u(s,x)\mathcal{L}^* \phi(s,x)\dd x\dd s\\
       &  \quad +\int_{\mathbb{R}^d}u(t,x)\phi(t,x)\dd x-\int_{\mathbb{R}^d}h(x)\phi(T,x)\dd x\\
       & =\int_t^T\!\!\!\int_{\mathbb{R}^d}f(s,x,u(s,x),(\nabla u\sigma)(s,x))\phi(s,x)\dd x\dd s +\int_t^T\!\!\!\int_{\mathbb{R}^d}g(s,x,u(s,x))\phi(s,x)\dd x\dd B_s,
   \end{split}
\end{equation*}
where $\mathcal{L}^*$ is the adjoint of $\mathcal{L}$.
\end{defn}
We close this section by recalling the following wellposedness and regularity results on BDSDEs.
\begin{prop}[\cite{MR3538011,MR1822898}]\label{pr}
Under the assumptions \textbf{(H1)-(H3)}, there exists a unique weak solution $u\in\mathcal{H}$ of the SPDE \eqref{spde}. Moreover, $u(t, x) = Y^{t,x}_t$, where $(Y^{t,x}_s,Z^{t,x}_s)_{t\leq s\leq T}$ is the solution of the BDSDE \eqref{bdsde}. Furthermore, we have for all $s \in [t, T ]:$
\begin{equation*}
  Y^{t,x}_s=u(s,X^{t,x}_s),~Z^{t,x}_s=(\nabla u\sigma)(s,X^{t,x}_s).
\end{equation*}
\end{prop}
\begin{prop}[\cite{MR3538011,MR3356530,MR1258986}]\label{regularity}
Let \textbf{(H1)-(H3)} hold. Then, for $0\leq r\leq s\leq T$, there exists a positive constant $C$ such that
\begin{equation*}
\begin{split}
   & E|Y_s^{t,x}-Y_r^{t,x}|^2 \leq C(1+|x|^2)(s-r), \\
   & E|Z_s^{t,x}-Z_r^{t,x}|^2\leq C(1+|x|^2)(s-r).
\end{split}
\end{equation*}
\end{prop}

\section{A predictor-corrector deep learning algorithm}
In this section, based on predictor-corrector procedure \eqref{sde2}-\eqref{pde},  we shall present our predictor-corrector deep learning algorithm. To this end, we shall first discuss the temporal semi-discretizations and its error estimates, and then present the deep learning algorithm in the physical domain.
\subsection{Temporal semi-discretizations}\label{discrete scheme}
We first construct the discrete time approximations of the solution $\psi$ in \eqref{sde2} and the solution $\varphi$ in \eqref{pde}, to obtain the approximating sequence $\{u_i(x)\}_{i=0}^N$ defined in \eqref{sequence}. We begin with the approximation of SDE  \eqref{sde2}. We recall the partition of $[0,T]:$
 $$\pi^N\colon0=t_0<t_1<\cdots<t_N=T.$$  We further define the maximum partition size by $|\pi|:=\max\limits_{i}\Delta t_i$. The Euler scheme for \eqref{sde2} yields
\begin{equation}\label{psi}
  \psi_{t_i}^\pi=\varphi_{t_{i+1}}^\pi+g\left(t_{i+1},X_{t_{i+1}}^\pi,\varphi_{t_{i+1}}^\pi\right)\Delta B_i,
\end{equation}
where $\varphi_{t_{i+1}}^\pi$ and $X_{t_{i+1}}^\pi$ are approximations of $\varphi(t,x)$ and $X_t$ at $t=t_{i+1}$ which will be clarified later. In each interval $[t_i,t_{i+1})$, the PDE \eqref{pde} is solved via its equivalent BSDE
\begin{equation}\label{bsde}
  y_s^{t,x}=\psi_{t_{i}}(x)+\int_{s}^{t_{i+1}}\!\!\! f\left(r,X_r^{t,x},y_r^{t,x},z_r^{t,x}\right)\dr-\int_{s}^{t_{i+1}}\!\!\! z_r^{t,x}\dwr,~t_i\leq s< t_{i+1}.
\end{equation}
It should be noted that the family $\{(y,z)\}$ differs from $\{(Y,Z)\}$ in \eqref{bdsde}, and $y$ is discontinuous at points $t_{i}.$ We recall Theorem 3.2 in \cite{MR1176785} that shows
\begin{equation}\label{bsde-rep}
  \varphi(t,x)=y^{t,x}_t \quad \text{and} \quad  \nabla \varphi(t,x)\sigma(t,x)=z^{t,x}_t.
\end{equation}
For this BSDE, we approximate the forward process $X_t$ by the Euler scheme
\begin{equation}\label{exi}
  X^{t,x,\pi}_{t_{i+1}} =X^{t,x,\pi}_{t_i}+\mu\left(t_i,X^{t,x,\pi}_{t_i}\right)\Delta t_i+\sigma\left(t_i,X^{t,x,\pi}_{t_i}\right)\Delta W_i,
\end{equation}
Hereafter, the superscript $t,x$ will be dropped unless clarity is needed. For the above scheme, the following estimate holds (see, e.g., \cite{MR2578878}).
\begin{lem}\label{semi-sde}
Assume that $\mu$ and $\sigma$ satisfy \textbf{(H3)}. Then there exists a constant $C$, independent of $\pi$, such that
\begin{equation*}
       \max\limits_{0\leq i\leq N-1}\sup\limits_{s\in[t_i,t_{i+1})} \mathbb{E}\left[|X_{s}-X_{t_i}^\pi|^2+|X_s-X_{t_{i+1}}^\pi|^2\right]<C(1+|x|^2)|\pi|.
\end{equation*}
\end{lem}
Next, the solution $(y,z)$ of \eqref{bsde} at $t=t_i$ is approximated by the solution $(y_{t_i}^{\pi},z_{t_i}^{\pi})$ via the following equation:
\begin{equation}\label{semidiscrete}
\begin{split}
    y_{t_i}^{\pi}&=\psi^{\pi}_{t_i}+f\left(t_i,X_{t_{i}}^\pi,y^{\pi}_{t_{i}},z_{t_{i}}^{\pi}\right)\Delta t_i-z_{t_i}^{\pi}\Delta W_i.
\end{split}
\end{equation}
Notice that the scheme in \eqref{semidiscrete} is implicit for $y$ and $z$. Thus, by \eqref{bsde-rep}, $\varphi({t_i},\cdot)$ is approximated by
\begin{equation*}
  \varphi_{t_i}^\pi=y_{t_i}^{\pi}.
\end{equation*}
To sum up, given $\{(X_{t_i}^\pi,y^{\pi}_{t_{i}},z_{t_i}^{\pi})\}_{i=0}^{N}$, the solution $\psi$ and $\varphi$ of \eqref{sde2} and \eqref{pde} is approximated by $\{\psi_{t_i}^\pi\}_{i=0}^N$ and $\{\varphi_{t_i}^\pi\}_{i=0}^N$, respectively, defined in a backward manner by $\varphi^\pi_{t_N}=h(X^{\pi}_{t_N})$,  and
\begin{equation}\label{semilinear}
   \begin{split}
   &\psi_{t_i}^\pi=\varphi_{t_{i+1}}^\pi+g(t_{i+1},X_{t_{i+1}}^\pi,\varphi_{t_{i+1}}^\pi)\Delta B_i,\\
     &  \varphi_{t_i}^\pi=y^{\pi}_{t_i}.
   \end{split}
\end{equation}
Notice that the approximate sequence $\{u_i\}_{i=0}^N$ in \eqref{sequence} is defined by
\begin{equation}\label{discretization}
  u_i(x)=\varphi_{t_i}^\pi(x), \quad i=N-1,\cdots,0.
\end{equation}

\subsection{Error estimations of temporal semi-discretizations}
In this section, we analyze the approximation error of the discrete time scheme described in Section \ref{discrete scheme}. We first investigate the numerical error between the solution of BDSDE \eqref{bdsde} and the discrete time scheme \eqref{semidiscrete}. To begin, we first present the following proposition.

\begin{prop}\label{th1}
Suppose that \textbf{(H1)}-\textbf{(H3)} hold. Then there exists a constant $C$, independent of $\pi$, $k$, $l$, and the space dimension $d$,  such that

\begin{equation}%\label{}
  \max\limits_{0\leq i\leq N-1}\mathbb{E}|Y_{t_i}^{t_i,\cdot}-y_{t_i}^\pi|^2\leq C(1+|x|^2)|\pi|.
\end{equation}
\end{prop}
\begin{proof}
Let $\mathbb{E}_i[\cdot]$ be the conditional expectation with respect to the discrete filtration $\mathcal{G}^\pi_{t_i}$ defined by
\begin{equation*}
  \mathcal{G}^\pi_{t_i}:=\sigma\{\Delta W_j;0\leq j\leq i-1\}\vee \sigma\{\Delta B_j;0\leq j\leq N-1\}.
\end{equation*}
Setting $s=t_i$, taking the conditional expectation $\E[\cdot]$ in \eqref{bdsde} and using the tower property, we have
\begin{equation}\label{yi}
   \begin{split}
     Y_{t_i} & =\E\left[ h(X_T)+\int_{t_i}^T f(r,X_r,Y_r,Z_r)\dr+\int_{t_i}^T  g(r,X_r,Y_r)\dd B_r\right]\\
       &=\E\left[Y_{t_{i+1}}+\int_{t_i}^{t_{i+1}}f(r,X_r,Y_r,Z_r)\dr+\int_{t_i}^{t_{i+1}} g(r,X_r,Y_r)\dd B_r\right].
   \end{split}
\end{equation}
We now define $$\Delta {y}_i:=Y_{t_i}-{y}_{t_i}^\pi(=Y_{t_i}-\varphi_{t_i}^\pi), \quad \Delta {z}_i:=Z_{t_i}-{z}_{t_i}^\pi, \quad i=0,\cdots,N-1.$$
Then, by \eqref{yi}, \eqref{semidiscrete} and \eqref{psi}, we have
\begin{equation*}
   \begin{split}
      \Delta {y}_i & = \E[Y_{t_{i+1}}-\psi^{\pi}_{t_i}]+\E\Bigg[\int_{t_i}^{t_{i+1}}\delta {f}^i(r)\dr\Bigg]+\E\Bigg[\int_{t_i}^{t_{i+1}}g(r,X_r,Y_r)\dbr\Bigg]\\
&=\E[\Delta {y}_{i+1}]+\E\Bigg[\int_{t_i}^{t_{i+1}}\delta {f}^i(r)\dr\Bigg]+\E\Bigg[\int_{t_i}^{t_{i+1}}\delta{g}^{i+1}(r)\dbr\Bigg],
   \end{split}
\end{equation*}
where
\begin{equation*}
   \begin{split}
       & \delta {f}^i(r):={f}(r,X_r,Y_r,Z_r)-f(t_i,X_{t_{i}}^\pi,y^{\pi}_{t_{i}},z_{t_{i}}^{\pi}), \\
       & \delta {g}^{i+1}(r):={g}(r,X_r,Y_r)-g(t_{i+1},X_{t_{i+1}}^\pi,y_{t_{i+1}}^\pi).
   \end{split}
\end{equation*}
By squaring and then taking the expectation on both sides of the above equation we obtain
\begin{equation*}
  \begin{split}
&\mathbb{E}|\Delta {y}_i|^2=\mathbb{E}|\E[\Delta {y}_{i+1}]|^2+\mathbb{E}\Bigg|\int_{t_i}^{t_{i+1}}\!\!\E[\delta {f}^i(r)]\dr\Bigg|^2+\mathbb{E}\Bigg|\E\Bigg[\int_{t_i}^{t_{i+1}}\!\!\!\delta{g}^{i+1}(r)\dbr\Bigg]\Bigg|^2\\
&\qquad\qquad +2\mathbb{E}\Bigg[\E[\Delta {y}_{i+1}]\int_{t_i}^{t_{i+1}}\!\!\!\E[\delta {f}^i(r)]\dr\Bigg]+2\mathbb{E}\Bigg[\E[\Delta {y}_{i+1}]\E\Big[\int_{t_i}^{t_{i+1}}\!\!\!\delta{g}^{i+1}(r)\dbr\Big]\Bigg]\\
&\qquad\qquad+2\mathbb{E}\Bigg[\int_{t_i}^{t_{i+1}}\!\!\!\E[\delta {f}^i(r)]\dr\E\Big[\int_{t_i}^{t_{i+1}}\!\!\!\delta{g}^{i+1}(r)\dbr\Big]\Bigg]:=I_1+I_2+I_3+I_4+I_5+I_6.
\end{split}
\end{equation*}
By Cauchy's inequality, Jensen's inequality, and \textbf{(H1)}, we have
\begin{equation*}
   \begin{split}
       I_2&\leq  \Delta t_i \mathbb{E}\Big[\int_{t_i}^{t_{i+1}}\!\!\! \left( \E[\delta f^i(r)]\right)^2\dr\Big] \\
       &\leq  K\Delta t_i \int_{t_i}^{t_{i+1}} \!\!\!\left(|r-t_{i}|+\mathbb{E}|X_r -X_{t_i}^\pi|^2+\mathbb{E}|Y_r-y_{t_i}^\pi|^2+\mathbb{E}|Z_r-z_{t_{i}}^\pi|^2\right)\dr.
   \end{split}
\end{equation*}
Next, by It\^{o}'s isometry, Jensen's inequality, and \textbf{(H1)}, we obtain
\begin{equation*}
   \begin{split}
       I_3&=  \mathbb{E}\Big[\int_{t_i}^{t_{i+1}}\!\!\! \left(\E[\delta{g}^{i+1}(r)]\right)^2\dr\Big] \\
       &\leq K \int_{t_i}^{t_{i+1}}\!\!\!\Big(|r-t_{i}|+\mathbb{E}|X_r -X_{t_{i+1}}^\pi|^2+\mathbb{E}|Y_r-y_{t_{i+1}}^\pi|^2\Big)\dr.
   \end{split}
\end{equation*}
Notice that $Y_s$ is $\mathcal{F}_{s,T}^B$-measurable, $\Delta {y}_{i+1}$ is $\mathcal{F}_{t_{i+1},T}^B$-measurable, thus we have
\begin{equation*}
  \mathbb{E}\Bigg[\E[\Delta {y}_{i+1}]\E\Big[\int_{t_i}^{t_{i+1}}\!\!\!{g}(r,X_r,Y_r)-g(t_{i+1},X_{t_{i+1}},Y_{t_{i+1}})\dbr\Big]\Bigg]=0.
\end{equation*}
Consequently, by properties of the conditional expectations we obtain
\begin{equation*}
\begin{split}
  I_5 & =\mathbb{E}\left[\E[\Delta {y}_{i+1}]\E\Big[\tilde{\delta} g^{i+1}\Delta B_i\Big]\right]=0,
\end{split}
\end{equation*}
where $\tilde{\delta} g^{i+1}=g(t_{i+1},X_{t_{i+1}},Y_{t_{i+1}})-g(t_{i+1},X_{t_{i+1}}^\pi,y_{t_{i+1}}^\pi)$. We next estimate $I_4$ and $I_6$. It follows from Young's inequality, Jensen's inequality, and the estimates of $I_2$ and $I_3$ that
\begin{equation*}
\begin{aligned}
I_4+I_6
 & \leq {\epsilon_1}\Delta t_i\mathbb{E}|\Delta {y}_{i+1}|^2+\frac{1}{\epsilon_1\Delta t_i}\mathbb{E}\Bigg|\int_{t_i}^{t_{i+1}}\!\!\!\E[\delta {f}^i(r)]\dr\Bigg|^2\\
 & \quad +\frac{1}{\epsilon_2}\mathbb{E}\Bigg|\int_{t_i}^{t_{i+1}}\!\!\!\E[\delta {f}^i(r)]\dr\Bigg|^2+{\epsilon_2}\mathbb{E}\Bigg|\int_{t_i}^{t_{i+1}}\!\!\! \E[\delta{g}^{i+1}(r)]\dbr\Bigg|^2\\
        &\leq  {\epsilon_1}\Delta t_i\mathbb{E}|\Delta {y}_{i+1}|^2+{\epsilon_2}K \int_{t_i}^{t_{i+1}} \!\!\!\Big(|r-t_{i}|+\mathbb{E}|X_r -X_{t_{i+1}}^\pi|^2+\mathbb{E}|Y_r-y_{t_{i+1}}^\pi|^2\Big)\dr\\
 &\quad   +    \left(\frac{1}{\epsilon_1}+\frac{\Delta t_i}{\epsilon_2}\right)K \int_{t_i}^{t_{i+1}}\!\!\! \Big(|r-t_{i}|
\mathbb{E}|X_r -X_{t_i}^\pi|^2+\mathbb{E}|Y_r-y_{t_i}^\pi|^2+\mathbb{E}|Z_r-z_{t_{i}}^\pi|^2\Big)\dr.
\end{aligned}
\end{equation*}
By Jensen's inequality, we have
\begin{equation*}
   \begin{split}
       &\mathbb{E}|Y_r-y_{t_i}^\pi|^2\leq 2\mathbb{E}|\Delta {y}_{i}|^2+2\mathbb{E}|Y_r-Y_{t_{i}}|^2, \\
       &\mathbb{E}|Y_r-y_{t_{i+1}}^\pi|^2\leq 2\mathbb{E}|\Delta {y}_{i+1}|^2+2\mathbb{E}|Y_r-Y_{t_{i+1}}|^2,\\
       &\mathbb{E}|Z_r-z_{t_{i}}^\pi|^2\leq 2\mathbb{E}|\Delta {z}_{i}|^2+2\mathbb{E}|Z_r-Z_{t_{i}}|^2.
   \end{split}
\end{equation*}
Then, using together Proposition \ref{regularity}, Lemma \ref{semi-sde},  and \textbf{(H1)} we obtain
\begin{equation}\label{eyi}
  \begin{split}
\mathbb{E}|\Delta {y}_i|^2&\leq\mathbb{E}|\E[\Delta {y}_{i+1}]|^2+{\epsilon_1}\Delta t_i\mathbb{E}|\Delta {y}_{i+1}|^2+2K(1+{\epsilon_2})\Delta t_i\mathbb{E}|\Delta {y}_{i+1}|^2\\
  &\quad + 2K\left(\frac{1}{\epsilon_1}+\Delta t_i+\frac{\Delta t_i}{\epsilon_2}\right)(\Delta t_i) \mathbb{E}|\Delta {y}_{i}|^2\\
  &\quad + 2K\left(\frac{1}{\epsilon_1}+\Delta t_i+\frac{\Delta t_i}{\epsilon_2}\right)(\Delta t_i) \mathbb{E}|\Delta {z}_{i}|^2+C(1+|x|^2)(\Delta t_i)^2.
\end{split}
\end{equation}

Now we turn to the estimation of $\Delta {z}_i$. Setting $s=t_i$, multiplying \eqref{bdsde} by $\Delta W_i$, and taking the conditional expectation $\E[\cdot]$, we obtain by the integration by parts formula that
\begin{equation}\label{zi}
   \begin{split}
\Delta t_i  Z_{t_i}&=\E[Y_{t_{i+1}}\Delta W_i]+\E\left[\int_{t_i}^{t_{i+1}}\!\!\!{f}(r,X_r,Y_r,Z_r)\dr\Delta W_i\right]\\
      &\quad +\E\left[\int_{t_i}^{t_{i+1}}\!\!\!g(r,X_r,Y_r)\dbr\Delta W_i\right]+\int_{t_i}^{t_{i+1}}\!\!\! \E[Z_{t_i}- Z_{r}]\dr.
   \end{split}
\end{equation}
Similarly, multiplying \eqref{semidiscrete} by $\Delta W_i$, taking the conditional expectation $\E[\cdot]$, and combining with \eqref{psi}, we get
\begin{equation}\label{semidiscrete-z}
\begin{split}
    {\Delta t_i}z_{t_i}^{\pi}&= \E\left[\psi^{\pi}_{t_i}{\Delta W_i}\right]=\E\left[(\varphi_{t_{i+1}}^\pi+g(t_{i+1},X_{t_{i+1}}^\pi,\varphi_{t_{i+1}}^\pi)\Delta B_i){\Delta W_i}\right].
\end{split}
\end{equation}
Subtracting \eqref{semidiscrete-z} from \eqref{zi} gives
\begin{equation*}
   \begin{split}
\Delta t_i\Delta {z}_i
&=\E[\Delta {y}_{i+1}\Delta W_i]+\E\Bigg[\int_{t_i}^{t_{i+1}}\!\!\!{f}(r,X_r,Y_r,Z_r)\dr\Delta W_i\Bigg] \\
&\quad +\E\Bigg[\int_{t_i}^{t_{i+1}}\!\!\! \delta{g}^{i+1}(r)\dbr\Delta W_i\Bigg]+\int_{t_i}^{t_{i+1}}\!\!\! \E[Z_{t_i}- Z_{r}]\dr.
   \end{split}
\end{equation*}
We now square the above equation and take expectation to get
\begin{equation*}
   \begin{split}
     \mathbb{E}\big|\Delta t_i\Delta {z}_i\big|^2
     &= \mathbb{E}\Big|\E[\Delta {y}_{i+1}\Delta W_i]\Big|^2\\
     &\quad+\mathbb{E}\Bigg[\left\{\E\Big[\int_{t_i}^{t_{i+1}}\!\!\!{f}(r,X_r,Y_r,Z_r)\dr\Delta W_i\Big]\right.\\
     &\quad\quad\quad+\left.\E\Big[\int_{t_i}^{t_{i+1}}\!\!\! \delta{g}^{i+1}(r)\dbr\Delta W_i\Big]+\int_{t_i}^{t_{i+1}} \!\!\!\E[Z_{t_i}- Z_{r}]\dr\right\}^2\Bigg]\\
     &\quad+2\mathbb{E}\Bigg[\E[\Delta {y}_{i+1}\Delta W_i]\E\Big[\int_{t_i}^{t_{i+1}}\!\!\!{f}(r,X_r,Y_r,Z_r)\dr\Delta W_i\Big]\Bigg]\\
     &\quad+2\mathbb{E}\Bigg[\E[\Delta {y}_{i+1}\Delta W_i]\E\Big[\int_{t_i}^{t_{i+1}} \!\!\!\delta{g}^{i+1}(r)\dbr\Delta W_i\Big]\Bigg]\\
     &\quad+2\mathbb{E}\Bigg[\E[\Delta {y}_{i+1}\Delta W_i]\int_{t_i}^{t_{i+1}} \!\!\!\E[Z_{r}- Z_{t_i}]\dr\Bigg]\\
     &:=J_1+J_2+J_3+J_4+J_{5}.
   \end{split}
\end{equation*}
We now estimate term by term. By Cauchy's inequality, Jensen's inequality, It\^{o}'s isometry, \textbf{(H1)}, Proposition \ref{regularity}, Lemma \ref{semi-sde}, and the estimate of $I_3$, we have
\begin{equation*}
   \begin{split}
     J_2 & \leq 3(\Delta t_i)^2\int_{t_i}^{t_{i+1}}\!\!\!\mathbb{E}|{f}(r,X_r,Y_r,Z_r)|^2\dr\\
       &\quad +3\Delta t_i\int_{t_i}^{t_{i+1}}\!\!\!\mathbb{E}|\delta{g}^{i+1}(r)|^2\dr+3\Delta t_i\int_{t_i}^{t_{i+1}}\!\!\!\mathbb{E}|Z_{t_i}- Z_{r}|^2\dr\\
       & \leq 6K(\Delta t_i)^2\mathbb{E}|\Delta {y}_{i+1}|^2+C(1+|x|^2)(\Delta t_i)^3.
   \end{split}
\end{equation*}
Combining the above result with Young's inequality, we have
\begin{equation*}
   \begin{split}
     J_3 & \leq{\varepsilon_3}\mathbb{E}\Big|\E[\Delta {y}_{i+1}\Delta W_i]\Big|^2+C(1+|x|^2)(\Delta t_i)^3,\\
     J_4 & \leq {\varepsilon_4}\mathbb{E}\Big|\E[\Delta {y}_{i+1}\Delta W_i]\Big|^2+\frac{1}{\varepsilon_4}\Big[2K(\Delta t_i)^2\mathbb{E}|\Delta {y}_{i+1}|^2+C(1+|x|^2)(\Delta t_i)^3\Big],\\
     J_5 & \leq {\varepsilon_5}\mathbb{E}\Big|\E[\Delta {y}_{i+1}\Delta W_i]\Big|^2+ C(1+|x|^2)(\Delta t_i)^3.
   \end{split}
\end{equation*}
Thus we have
\begin{equation}\label{1}
   \begin{split}
      (\Delta t_i)^2\mathbb{E}|\Delta {z}_i|^2 & \leq (1+{\varepsilon})\mathbb{E}\Big|\E[\Delta {y}_{i+1}\Delta W_i]\Big|^2 +C_z(\Delta t_i)^2\mathbb{E}|\Delta {y}_{i+1}|^2\\
       &\quad +C(1+|x|^2)(\Delta t_i)^3,
   \end{split}
\end{equation}
where ${\varepsilon}:={\varepsilon_3}+{\varepsilon_4}+{\varepsilon_5}$, and $C_z=6K+\frac{2K}{\varepsilon_4}$. Now by Cauchy's inequality we have
\begin{equation}\label{2}
  \Big|\E[\Delta {y}_{i+1}\Delta W_i]\Big|^2\leq \Delta t_i\Big(\E[|\Delta {y}_{i+1}|^2]-|\E[\Delta {y}_{i+1}]|^2\Big).
\end{equation}
Then plugging \eqref{2} and \eqref{1} into \eqref{eyi}, we obtain
\begin{equation}
\begin{split}
C_y^1\mathbb{E}|\Delta {y}_i|^2
&\leq C_y^2 \mathbb{E}|\Delta {y}_{i+1}|^2+({1-C_y^2})\mathbb{E}\Big|\E[\Delta y_{i+1}]\Big|^2\\
&\quad +C_y^3 \Delta t_i\mathbb{E}|\Delta {y}_{i+1}|^2+C\Big(1+|x|^2\Big)(\Delta t_i)^2,
\end{split}
\end{equation}
where $$C_y^1=1-2K\left(\frac{1}{\epsilon_1}+\Delta t_i+\frac{\Delta t_i}{\epsilon_2}\right)(\Delta t_i),\quad C_y^2={2K\frac{1+\varepsilon}{\varepsilon_1}},$$  $$C_y^3=\varepsilon_1+2K(1+\varepsilon_2)+2K\left(1+\frac{1}{\varepsilon_2}\right)(1+\varepsilon)+2K\left(\frac{1}{\epsilon_1}+\Delta t_i+\frac{\Delta t_i}{\epsilon_2}\right)C_z.$$
With appropriately chosen $\varepsilon_1$, $\varepsilon_3$, $\varepsilon_4$, and $\varepsilon_5$ we may set $C_y^2=1.$ Thus we obtain
\begin{align}\label{4}
\mathbb{E}|\Delta {y}_i|^2&\leq (C_y^1)^{-1} \mathbb{E}\Big|\Delta {y}_{i+1}\Big|^2+(C_y^1)^{-1}C_y^3 \Delta t_i\mathbb{E}\Big|\Delta {y}_{i+1}\Big|^2+C\Big(1+|x|^2\Big)(\Delta t_i)^2 \nonumber\\
&\leq \Big(1+C \Delta t_i\Big)\mathbb{E}\Big|\Delta {y}_{i+1}\Big|^2+C\Big(1+|x|^2\Big)(\Delta t_i)^2.
\end{align}
Then, by applying the discrete Gronwall inequality and Lemma \ref{semi-sde} to \eqref{4}, we have
\begin{equation*}
  \max\limits_{0\leq i\leq N}\mathbb{E}|\Delta {y}_i|^2\leq C(1+|x|^2)|\pi|.
\end{equation*}
This completes the proof.
\end{proof}
We now present the following theorem that shows the convergence of the time-discretization scheme.
\begin{thm}\label{semibound}
Let \textbf{(H1)}-\textbf{(H3)} hold. Then there exists a constant $C$, independent of $\pi$, $k$, $l$, and the space dimension $d$,  such that
   \begin{equation*}%\label{}
   \max\limits_{0\leq i\leq N-1}\mathbb{E}\Big|u(t_i,X_{t_i}^{t,x})-u_i\Big|_{\omega}^2\leq C|\pi|.
   \end{equation*}
\end{thm}
\begin{proof}
It follows from Proposition \ref{th1} that
\begin{equation*}
   \begin{split}
   \mathbb{E}\left[\int_{\mathbb{R}^d}\Big|u(t_i,X_{t_i}^{t,x})-u_i(x)\Big|^2\omega(x)\dd x\right]&=\mathbb{E}\left[\int_{\mathbb{R}^d}\Big|Y_{t_i}^{t_i,X_{t_i}^{t,x}}-y_i^\pi\Big|^2\omega(x)\dd x\right]\\
      &\leq \int_{\mathbb{R}^d}\Bigg(\max\limits_{0\leq i\leq N-1}\mathbb{E}\left|Y_{t_i}^{t_i,X_{t_i}^{t,x}}-y_i^\pi\right|^2\Bigg)\omega(x)\dd x\\
      &\leq C|\pi|\int_{\mathbb{R}^d}\Big(1+|x|^2\Big)\omega(x)\dd x\leq C|\pi|.
   \end{split}
\end{equation*}
The proof is completed.
\end{proof}

\subsection{DNNs-approximations in the physical domain}
In this section, we propose a deep learning-based numerical scheme for the discrete time approximations \eqref{semilinear} and \eqref{discretization}. To this end, we begin with a general review of how deep neural networks (DNNs) approximate unknown functions. Consider a feedforward neural network $\mathcal{N}_{L+1}$ with $L+1$ layers ($L>1$). The input layer will be referred to as layer 0 and the output layer as layer $L$. The layers in between are called \emph{hidden layers}. The basic building block of a neural network is an \emph{artificial neuron} or \emph{node}. Each input $i_k$ has an associated weight $w_k$. For the sake of simplicity, the bias terms are not treated specially, as they correspond to a weight with a fixed input of 1.  The sum of all weighted inputs, $S=\sum_k i_kw_k$, is passed through a nonlinear activation function $\sigma$, to transform the preactivation level of the neuron to an output $O=\sigma(\sum_k i_kw_k)$. The output $O$ is used as an input by nodes in the next layer. Let $n_l$ be the number of neurons on the $l$-th layer. For simplicity, we set $n_l=\bar{n}$, for $l=1,\cdots,L-1$. Thus a feedforward neural network is a function defined as the composition
\begin{equation*}
  \mathcal{N}_{L+1}^{n_0,\bar{n},n_L}(x_0,\mathcal{W},\sigma_0,\cdots,\sigma_L):=\sigma_L\circ S_L\circ\sigma_{L-1}\circ S_{L-1}\circ\cdots\circ\sigma_0\circ S_0(x_0),
\end{equation*}
where $x_0\in\mathbb{R}^{n_0}$ is the network input, $S_l$ and $\sigma_l$ are the weighted sum and the activation function of the $l$-th  layer, respectively. $\mathcal{W}$ contains all the parameters (weights and biases) of the network.  The dimension of
$\mathcal{W}$ is $$n_{\mathcal{W}}=\sum_{l=1}^L n_l\times n_{l-1}+n_{l-1}=(n_0+1)\bar{n}+(\bar{n}+1)\bar{n}(L-1)+(\bar{n}+1)n_L.$$

Recall the Euler method \eqref{exi}. For every fixed path $\{\Delta B_i,~0\leq i\leq N-1\}$ of the Brownian motion $B$, we approximate $\psi_{t_i}^\pi$ in \eqref{psi}  by
\begin{equation*}
  \psi_{t_i}^\pi\approx \Psi^*_{i}(X^{\pi}_{t_{i+1}},\Delta B_i),
\end{equation*}
with
\begin{equation}\label{psischeme}
  \Psi^*_{i}(X^{\pi}_{t_{i+1}},\Delta B_i)=\mathcal{U}^*_{i+1}(X^{\pi}_{t_{i+1}})+g\Big(t_{i+1},X_{t_{i+1}}^\pi,\mathcal{U}^*_{i+1}(X^{\pi}_{t_{i+1}})\Big)\Delta B_i,
\end{equation}
where the notation $\mathcal{U}^*_{i+1}(X^{\pi}_{t_{i+1}})$ will be explained later. For $i=N-1,\cdots,0$, we approximate $y_{t_i}^\pi$ and $z_{t_i}^\pi$ in \eqref{semidiscrete} by a pair of neural networks $\mathcal{U}_i(\cdot;\theta_{i})\in \mathcal{N}_{L+1}^{d,\bar{n},1}$ and $\mathcal{V}_i(\cdot;\theta_{i})\in \mathcal{N}_{L+1}^{d,\bar{n},d}$ with parameter $\theta_{i}$. Then, the optimal parameters $\theta^{*}_{i}$ are determined by minimizing the following loss function
 \begin{equation*}
 \begin{split}
L_{i}(\theta_{i})&=\mathbb{E}\Big|\Psi^*_{i}(X^{\pi}_{t_{i+1}},\Delta B_i)-\mathcal{V}_{i}(X^{\pi}_{t_{i}};\theta_i)\Delta W_i\\
&\quad +f\Big(t_{i},X^{\pi}_{t_i},\mathcal{U}_{i}(X^{\pi}_{t_i};\theta_i),\mathcal{V}_{i}(X^{\pi}_{t_{i}};\theta_i)\Big)\Delta t_i-\mathcal{U}_{i}(X^{\pi}_{t_i};\theta_{i})\Big|^2.
 \end{split}
\end{equation*}
We set $$\mathcal{U}^*_{i}(X^{\pi}_{t_i})=\mathcal{U}_{i}(X^{\pi}_{t_i};\theta^{*}_{i}), \quad \mathcal{V}^*_{i}(X^{\pi}_{t_i})=\mathcal{V}_{i}(X^{\pi}_{t_i};\theta^{*}_{i}).$$ It follows from \eqref{semilinear} that the value of the solution $u$ of \eqref{spde} at time $t_i$ is approximated by $\mathcal{U}^*_{i}(X^{\pi}_{t_i})$. We remark $\mathcal{V}^*_{i}(X^{\pi}_{t_i})$ is not an approximation of $(\nabla u\sigma)$ at time $t_i$. In comparison to the global optimization problem formulated in \cite{teng2021solving}, which requires to keep in memory all the computed network approximations of $u$, the memory requirements for our proposed algorithm are significantly reduced. We finally present the whole computational procedure in Algorithm 1.

\begin{algorithm}
\DontPrintSemicolon
\caption{}
  \KwData{Sample paths of $\{X^{\pi}_{t_i}\}_{i=0}^N$,~ $\{\Delta W_i\}_{i=0}^{N-1}$~ and~ $\{\Delta B_i\}_{i=0}^{N-1}$}
%  \KwInput{Input}
%  \KwOutput{Output}
  $\mathcal{U}^*_{N}= h(X^{\pi}_{t_N})$.%\tcp*{this is a comment}
  %\tcc{Now this is an if...else conditional loop}

  \For{$i=N-1,\cdots,0$}
  {
Set $\Psi^*_{i}(X^{\pi}_{t_{i+1}},\Delta B_i)=\mathcal{U}^*_{i+1}(X^{\pi}_{t_{i+1}})+g(t_{i+1},X_{t_{i+1}}^\pi,\mathcal{U}^*_{i+1}(X^{\pi}_{t_{i+1}}))\Delta B_i$

Train a pair of deep neural networks $\mathcal{U}_i(X^{\pi}_{t_i};\theta_{i})$ and $\mathcal{V}_i(X^{\pi}_{t_i};\theta_{i})$ for the approximation of $(y_{t_i}^\pi,z_{t_i}^\pi )$ by minimizing $L_{i}(\theta_{i})$. Set
\begin{equation*}
\theta^{*}_{i}=\arg\min L_{i}(\theta_{i}).
\end{equation*}

Update $\mathcal{U}^*_{i}(X^{\pi}_{t_i})=\mathcal{U}_{i}(X^{\pi}_{t_i};\theta^{*}_{i})$ and $\mathcal{V}^*_{i}(X^{\pi}_{t_i})=\mathcal{V}_{i}(X^{\pi}_{t_i};\theta^{*}_{i})$.

$u(0,X_{0})=Y_{0}^{0,X_{0}}\approx \mathcal{U}^*_{0}(X^{\pi}_{t_0})$.
  }
\end{algorithm}
We now define
\begin{equation}\label{discrete}
\begin{split}
    \mathcal{U}_{t_i}^{\pi}&=\Psi^*_{i}(X^{\pi}_{t_{i+1}},\Delta B_i)+f\Big(t_i,X_{t_{i}}^\pi,\mathcal{U}^{\pi}_{t_{i}},\mathcal{V}_{t_{i}}^{\pi}\Big)\Delta t_i-\mathcal{V}_{t_i}^{\pi}\Delta W_i,
\end{split}
\end{equation}
Then the following theorem presents the convergence result for the algorithm.
\begin{thm}\label{thm31}
Assume that \textbf{(H1)}-\textbf{(H3)} hold. We also suppose that the partition $\pi$ satisfies the regularity constraint $\frac{|\pi|}{\min\limits_{i}\Delta t_i}\leq c_0.$ Then there exists a constant $C$, independent of $\pi$, $k$, $l$, and the space dimension $d$, such that
\begin{equation*}
   \begin{split}
       & \quad \max\limits_{0\leq i\leq N-1}\mathbb{E}\Big|u(t_i,X_{t_i}^{t,x})-\mathcal{U}^*_{i}(X^{\pi}_{t_i})\Big|_\omega^2\\
&\leq C\left(|\pi|+\frac{c_0}{|\pi|}\sum_{i=0}^{N-1} \Bigg(\inf_{\theta_{i}}\mathbb{E}\Big|\mathcal{U}^{\pi}_{t_i}-\mathcal{U}_{i}(X^{\pi}_{t_i};\theta_{i})\Big|^2\Bigg)+\sum_{i=0}^{N-1} \Bigg(\inf_{\theta_{i}}\mathbb{E}\Big|\mathcal{V}^{\pi}_{t_i}-\mathcal{V}_{i}(X^{\pi}_{t_i};\theta_{i})\Big|^2\Bigg)\right).
   \end{split}
\end{equation*}
\end{thm}
\begin{proof}
From \eqref{yi}, \eqref{psischeme} and \eqref{discrete}, we have
\begin{equation*}%\label{}
   \begin{split}
Y_{t_i}-\mathcal{U}_{t_i}^{\pi} &= \E\Big[Y_{t_{i+1}}-\Psi^*_{i}(X^{\pi}_{t_{i+1}},\Delta B_i)\Big]+\E\left[\int_{t_i}^{t_{i+1}} \!\!\!g(r,X_r,Y_r)\dd B_r\right]\\
&\quad  +\E\left[\int_{t_i}^{t_{i+1}}\!\!\!f\Big(r,X_r,Y_r,Z_r\Big)\dr-f\Big(t_i,X_{t_i}^\pi,\mathcal{U}_{t_i}^{\pi},\mathcal{V}_{t_i}^{\pi}\Big)\Delta t_i\right]\\
&=\E\Big[Y_{t_{i+1}}-\mathcal{U}^*_{i+1}(X^{\pi}_{t_{i+1}})\Big]+\E\left[\int_{t_i}^{t_{i+1}}\!\!\! g(r,X_r,Y_r)\dd B_r-g\Big(t_{i+1},X_{t_{i+1}}^\pi,\mathcal{U}^*_{i+1}(X^{\pi}_{t_{i+1}})\Big)\Delta B_i\right]\\
       & \quad  +\E\left[\int_{t_i}^{t_{i+1}}\!\!\!f\Big(r,X_r,Y_r,Z_r\Big)\dr-f\Big(t_i,X_{t_i}^\pi,\mathcal{U}_{t_i}^{\pi},\mathcal{V}_{t_i}^{\pi}\Big)\Delta t_i\right].
   \end{split}
\end{equation*}
Using similar arguments as for Proposition \ref{th1} we have
\begin{equation*}
\mathbb{E}\Big|Y_{t_i}-\mathcal{U}^{\pi}_{t_i}\Big|^2 \leq(1+C \Delta t_i)\mathbb{E}\left|Y_{t_{i+1}}-\mathcal{U}^*_{i+1}(X^{\pi}_{t_{i+1}})\right|^2+C(1+|x|^2)(\Delta t_i)^2.
\end{equation*}
It follows from the Young's inequality that
\begin{equation*}
   \begin{split}
     \mathbb{E}\left|Y_{t_i}-\mathcal{U}^*_{i}(X^{\pi}_{t_i})\right|^2 & =\mathbb{E}\left|Y_{t_i}-\mathcal{U}^{\pi}_{t_i}+\mathcal{U}^{\pi}_{t_i}-\mathcal{U}^*_{i}(X^{\pi}_{t_i})\right|^2 \\
       & \leq (1+\Delta t_i)\mathbb{E}\left|Y_{t_i}-\mathcal{U}^{\pi}_{t_i}\right|^2+\left(1+\frac{1}{\Delta t_i}\right)\mathbb{E}\left|\mathcal{U}^{\pi}_{t_i}-\mathcal{U}^*_{i}(X^{\pi}_{t_i})\right|^2\\
&\leq(1+C \Delta t_i)\mathbb{E}\left|Y_{t_{i+1}}-\mathcal{U}^*_{i+1}(X^{\pi}_{t_{i+1}})\right|^2\\
&\quad +C(1+|x|^2)(\Delta t_i)^2+\left(1+\frac{1}{\Delta t_i}\right)\mathbb{E}\left|\mathcal{U}^{\pi}_{t_i}-\mathcal{U}^*_{i}(X^{\pi}_{t_i})\right|^2.
   \end{split}
\end{equation*}
Then by discrete Gronwall's inequality we obtain
\begin{equation}\label{gronwall}
\max\limits_{0\leq i\leq N-1}\mathbb{E}\left|Y_{t_i}-\mathcal{U}^*_{i}(X^{\pi}_{t_i})\right|^2\leq  C\sum_{i=0}^{N-1}\left(1+\frac{1}{\Delta t_i}\right)\mathbb{E}\left|\mathcal{U}^{\pi}_{t_i}-\mathcal{U}^*_{i}(X^{\pi}_{t_i})\right|^2+C\left(1+|x|^2\right)|\pi|.
\end{equation}
On the other hand by \eqref{discrete} we have
 \begin{equation*}
 \begin{split}
 L_{i}(\theta_{i})&=\mathbb{E}\left|\Psi^*_{i}(X^{\pi}_{t_{i+1}},\Delta B_i)-\mathcal{V}_{i}(X^{\pi}_{t_{i}};\theta_i)\Delta W_i+f\!\left(t_{i},X^{\pi}_{t_i},\mathcal{U}_{i}(X^{\pi}_{t_i};\theta_i),\mathcal{V}_{i}(X^{\pi}_{t_{i}};\theta_i)\right)\Delta t_i-\mathcal{U}_{i}(X^{\pi}_{t_i};\theta_{i})\right|^2\\
&=\mathbb{E}\Big|\mathcal{U}_{t_i}^{\pi}-f\!\left(t_i,X_{t_{i}}^\pi,\mathcal{U}^{\pi}_{t_{i}},\mathcal{V}_{t_{i}}^{\pi}\right)\Delta t_i+\mathcal{V}_{t_i}^{\pi}\Delta W_i-\mathcal{V}_{i}(X^{\pi}_{t_{i}};\theta_i)\Delta W_i  \\
&\quad \quad +f\!\left(t_{i},X^{\pi}_{t_i},\mathcal{U}_{i}(X^{\pi}_{t_i};\theta_i),\mathcal{V}_{i}(X^{\pi}_{t_{i}};\theta_i)\right)\Delta t_i- \mathcal{U}_{i}(X^{\pi}_{t_i};\theta_{i})\Big|^2\\
&=\mathbb{E}\Big|\mathcal{U}_{t_i}^{\pi}-\mathcal{U}_{i}(X^{\pi}_{t_i};\theta_{i})\Big|^2+\Delta t_i\mathbb{E}\Big|\mathcal{V}_{t_i}^{\pi}-\mathcal{V}_{i}(X^{\pi}_{t_{i}};\theta_i)\Big|^2\\
&\quad\quad  +(\Delta t_i)^2\mathbb{E}\Big|\Big(f(t_{i},X^{\pi}_{t_i},\mathcal{U}_{i}(X^{\pi}_{t_i};\theta_i),\mathcal{V}_{i}(X^{\pi}_{t_{i}};\theta_i)\Big)
-f(t_i,X_{t_{i}}^\pi,\mathcal{U}^{\pi}_{t_{i}},\mathcal{V}_{t_{i}}^{\pi})\Big|^2\\
&\quad\quad  +2\Delta t_i\mathbb{E}\Big[\Big(\mathcal{U}_{t_i}^{\pi}-\mathcal{U}_{i}(X^{\pi}_{t_i};\theta_{i})\Big)\Big(f(t_{i},X^{\pi}_{t_i},\mathcal{U}_{i}(X^{\pi}_{t_i};\theta_i),\mathcal{V}_{i}(X^{\pi}_{t_{i}};\theta_i)\Big)
-f(t_i,X_{t_{i}}^\pi,\mathcal{U}^{\pi}_{t_{i}},\mathcal{V}_{t_{i}}^{\pi})\Big)\Big]\\
&\leq (1+\Delta t_i)\mathbb{E}\Big|\mathcal{U}_{t_i}^{\pi}-\mathcal{U}_{i}(X^{\pi}_{t_i};\theta_{i})\Big|^2+\Delta t_i\mathbb{E}\Big|\mathcal{V}_{t_i}^{\pi}-\mathcal{V}_{i}(X^{\pi}_{t_{i}};\theta_i)\Big|^2\\
& \quad \quad +(1+1/\Delta t_i)(\Delta t_i)^2\mathbb{E}\Big|\Big(f(t_{i},X^{\pi}_{t_i},\mathcal{U}_{i}(X^{\pi}_{t_i};\theta_i),\mathcal{V}_{i}(X^{\pi}_{t_{i}};\theta_i)\Big)
-f(t_i,X_{t_{i}}^\pi,\mathcal{U}^{\pi}_{t_{i}},\mathcal{V}_{t_{i}}^{\pi})\Big|^2\\
&\leq \Big(1+\Delta t_i+K\Delta t_i+K(\Delta t_i)^2\Big)\mathbb{E}\Big|\mathcal{U}_{t_i}^{\pi}-\mathcal{U}_{i}(X^{\pi}_{t_i};\theta_{i})\Big|^2 +\Big(1+K+K\Delta t_i\Big)\Delta t_i\mathbb{E}\Big|\mathcal{V}_{t_i}^{\pi}-\mathcal{V}_{i}(X^{\pi}_{t_{i}};\theta_i)\Big|^2.
 \end{split}
\end{equation*}
Notice that $2K\Delta t_i a^2+\frac{1}{2K\Delta t_i}b^2+2ab=2K\Delta t_i\left(a+\frac{1}{2K\Delta t_i}b\right)^2\geq0$. Thus we have
\begin{equation*}
\begin{split}
L_{i}(\theta_{i})
&\geq (1-2K\Delta t_i)\mathbb{E}\Big|\mathcal{U}_{t_i}^{\pi}-\mathcal{U}_{i}(X^{\pi}_{t_i};\theta_{i})\Big|^2+\Delta t_i\mathbb{E}\Big|\mathcal{V}_{t_i}^{\pi}-\mathcal{V}_{i}(X^{\pi}_{t_{i}};\theta_i)\Big|^2\\
&\quad +\left(1-\frac{1}{2K\Delta t_i}\right)(\Delta t_i)^2\mathbb{E}\Big|\Big(f(t_{i},X^{\pi}_{t_i},\mathcal{U}_{i}(X^{\pi}_{t_i};\theta_i),\mathcal{V}_{i}(X^{\pi}_{t_{i}};\theta_i)\Big)
-f(t_i,X_{t_{i}}^\pi,\mathcal{U}^{\pi}_{t_{i}},\mathcal{V}_{t_{i}}^{\pi})\Big|^2\\
&\geq (1-2K\Delta t_i)\mathbb{E}\Big|\mathcal{U}_{t_i}^{\pi}-\mathcal{U}_{i}(X^{\pi}_{t_i};\theta_{i})\Big|^2+\Delta t_i\mathbb{E}\Big|\mathcal{V}_{t_i}^{\pi}-\mathcal{V}_{i}(X^{\pi}_{t_{i}};\theta_i)\Big|^2\\
&\quad -\frac{\Delta t_i}{2K}\mathbb{E}\Big|\Big(f(t_{i},X^{\pi}_{t_i},\mathcal{U}_{i}(X^{\pi}_{t_i};\theta_i),\mathcal{V}_{i}(X^{\pi}_{t_{i}};\theta_i)\Big)
-f(t_i,X_{t_{i}}^\pi,\mathcal{U}^{\pi}_{t_{i}},\mathcal{V}_{t_{i}}^{\pi})\Big|^2\\
&\geq \left(1-2K\Delta t_i-\frac{1}{2}\Delta t_i\right)\mathbb{E}\Big|\mathcal{U}_{t_i}^{\pi}-\mathcal{U}_{i}(X^{\pi}_{t_i};\theta_{i})\Big|^2+\frac{1}{2}\Delta t_i\mathbb{E}\Big|\mathcal{V}_{t_i}^{\pi}-\mathcal{V}_{i}(X^{\pi}_{t_{i}};\theta_i)\Big|^2.
  \end{split}
\end{equation*}
We take $\theta_i=\theta_i^*,$ and it is easy to verify that
\begin{equation*}
  \begin{split}
    \left(1-2K\Delta t_i-\frac{1}{2}\Delta t_i\right)\mathbb{E}\Big|\mathcal{U}_{t_i}^{\pi}-\mathcal{U}^*_{i}(X^{\pi}_{t_i})\Big|^2\leq L_{i}(\theta_{i}^*),
  \end{split}
\end{equation*}
which implies
\begin{equation*}
   \begin{split}
        &\quad \left(1-2K\Delta t_i-\frac{1}{2}\Delta t_i\right)\mathbb{E}\Big|\mathcal{U}_{t_i}^{\pi}-\mathcal{U}^*_{i}(X^{\pi}_{t_i})\Big|^2 \\
        &\leq L_{i}(\theta_{i})\leq (1+C\Delta t_i)\mathbb{E}\Big|\mathcal{U}_{t_i}^{\pi}-\mathcal{U}_{i}(X^{\pi}_{t_i};\theta_{i})\Big|^2+C\Delta t_i\mathbb{E}\Big|\mathcal{V}_{t_i}^{\pi}-\mathcal{V}_{i}(X^{\pi}_{t_{i}};\theta_i)\Big|^2.
   \end{split}
\end{equation*}
Substituting this into the right-hand side of \eqref{gronwall} yields
\begin{equation}
\begin{split}
&\quad \max\limits_{0\leq i\leq N-1}\mathbb{E}\Big|Y_{t_i}-\mathcal{U}^*_{i}(X^{\pi}_{t_i})\Big|^2\\
&\leq  C\sum_{i=0}^{N-1}\Bigg(\frac{c_0}{|\pi|}\mathbb{E}\Big|\mathcal{U}_{t_i}^{\pi}-\mathcal{U}_{i}(X^{\pi}_{t_i};\theta_{i})\Big|^2+\mathbb{E}\Big|
\mathcal{V}_{t_i}^{\pi}-\mathcal{V}_{i}(X^{\pi}_{t_{i}};\theta_i)\Big|^2\Bigg)\\
&\quad +C(1+|x|^2)|\pi|.
   \end{split}
\end{equation}
Consequently, we have
\begin{equation*}
   \begin{split}
&\quad \max\limits_{0\leq i\leq N-1}\mathbb{E}|Y_{t_i}-\mathcal{U}^*_{i}(X^{\pi}_{t_i})|^2\\
&\leq  C\sum_{i=0}^{N-1}\Bigg(\frac{c_0}{|\pi|}\inf_{\theta_{i}}\mathbb{E}\Big|\mathcal{U}_{t_i}^{\pi}-\mathcal{U}_{i}(X^{\pi}_{t_i};\theta_{i})\Big|^2+\inf_{\theta_{i}}\mathbb{E}
\Big|\mathcal{V}_{t_i}^{\pi}-\mathcal{V}_{i}(X^{\pi}_{t_{i}};\theta_i)\Big|^2\Bigg)\\
&\quad +C(1+|x|^2)|\pi|.
   \end{split}
\end{equation*}
The desired result follows by using an similar argument as in the proof of Theorem \ref{semibound}. The proof is complete.
\end{proof}

Theorem \ref{thm31} indicates that the approximation error of the proposed algorithm can be controlled by the neural network approximation errors $\mathbb{E}|\mathcal{U}_{t_i}^{\pi}-\mathcal{U}_{i}(X^{\pi}_{t_i};\theta_{i})|^2$ and $\mathbb{E}|\mathcal{V}_{t_i}^{\pi}-\mathcal{V}_{i}(X^{\pi}_{t_i};\theta_{i})|^2$. In addition, the number of parameters $n_{\mathcal{W}}$ used in the neural network grows at most polynomially in the space dimension $d$, similar as in \cite{MR4292849}.

\section{Numerical examples}

In this section, we shall present some numerical experiments. All our numerical tests are performed in Python using TensorFlow 2.7 on a laptop equipped with an Intel Core i5 Processor with 1.8GHz. We use fully connected layers and batch normalization \cite{pmlr-v37-ioffe15} after each matrix multiplication and before activation. We employ the Leaky Rectified Linear Unit function as the activation function for the hidden layers, and the identity function as the activation function for the output layer. For the optimization solver, we use Adam optimizer \cite{DBLP:journals/corr/KingmaB14} with the exponential decay rate of 0.9 for the first moment estimates, and the exponential decay rate of 0.999 for the second moment estimates. The training was on mini-batches with 64 trajectories of $X$  per batch for 100 epochs. All the weights in the network are initialized using the He initialization method \cite{He_2015_ICCV}. We use adaptive learning rate for the training process with a starting value of 0.01 and drop it by half if the loss doesn't decrease for 10 consecutive epochs.

We consider an example adapted from Section 4 in \cite{teng2021solving}, for which the parameters of the SPDE \eqref{spde} are chosen as: $k=1$, $l=1$, $\sigma=0.25$, $T=1$, $\mu(t,x)=\sigma\sqrt{d}\sin(x)$, $\sigma(t,x)=\sigma \sqrt{d}I_d$, $h(x)=\sqrt{d}\arctan(\frac{1}{d}\sum_{j=1}^d x_j)+\frac{\pi}{2}\sqrt{d}$, $f=-\sigma \sin(x)\nabla u$, and $g=-\sigma\sqrt{d}\sin^2(\frac{1}{\sqrt{d}}u)$. We also set $X_0\sim U(-0.2,0.2)$) to avoid overfitting. The associated exact solution reads:
\begin{equation*}
   \begin{split}
     u(t,X_t) & = \sqrt{d}\arctan\left(\frac{1}{d}\sum_{j=1}^d (X_t)_{j}-\sigma(B_T-B_t)\right)+\frac{\pi}{2}\sqrt{d}.
   \end{split}
\end{equation*}
We shall test our algorithm in four cases: the one-dimensional case, a moderate dimensional (5-dimensional) case, and two high-dimensional (50 and 100-dimensional) cases.

For the one-dimensional case, we discretize the equation using $N=32$ time steps. Each of the neural networks consists of 2 hidden layers with 11 nodes. The number of iterations per epoch is set to be 100. The simulation results are presented in Table 1. For each fixed sample path of $B$, we provide the average $\bar{\mathcal{U}}^*_{0}(X^{\pi}_{0})$ and the standard deviation of $\mathcal{U}^*_{0}(X^{\pi}_{t_0})$ by performing 5 independent runs of the algorithm. The relative error is calculated as $$\frac{|\bar{\mathcal{U}}^*_{0}(X^{\pi}_{t_0})-u(0,X_0)|}{u(0,X_0)}.$$ We show in Figs.1-5 a comparison between the exact solution and the approximate solutions at $t=0.03125$ for each realization of $B.$ The approximation errors are also presented in Table 1. It is clearly shown that the numerical solution match well with the exact solution.

\begin{table}[!ht]
    \centering
    \begin{tabular}{|l|l|l|l|}
    \hline
        Averaged Approx. & Exact solution & Standard deviation & Relative error \\ \hline
        1.5998764 & 1.596926235 & 0.004728534 & 0.001847402 \\ \hline
        1.9575781 & 1.967079301 & 0.002868398 & 0.004830106 \\ \hline
        1.8510513 & 1.847140899 & 0.001399888 & 0.002117002 \\ \hline
        1.8063872 & 1.808914966 & 0.004120083 & 0.001397394 \\ \hline
        1.4118652 & 1.403613807 & 0.00232844 & 0.005878677 \\ \hline
        \multicolumn{4}{|l|}{Relative $L^2$ error: 0.00368065398642483}  \\ \hline
    \end{tabular}
\caption{Numerical results for the one-dimensional case}
\end{table}
\begin{figure}[H]
  \centering
  \includegraphics[width=0.5\textwidth]{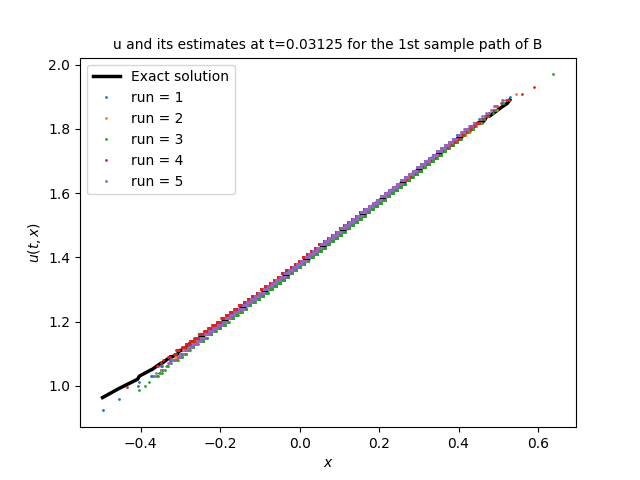}
  \caption{Comparison between the exact solution and the approximate solutions for the 1st sample path of $B$ ($d=1$)}
\end{figure}

\begin{figure}[H]
  \centering
  \includegraphics[width=0.5\textwidth]{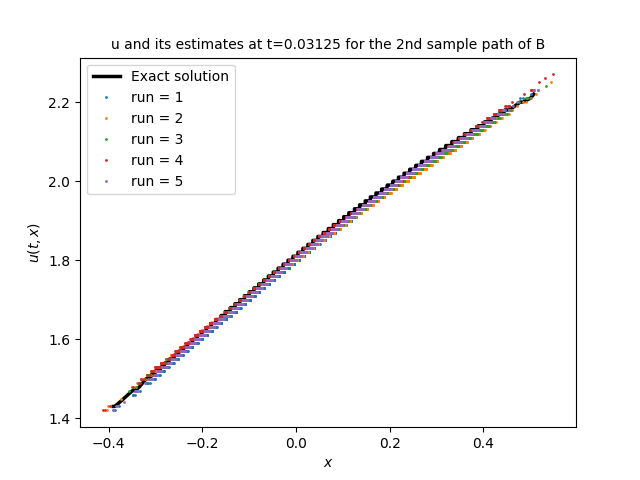}
  \caption{Comparison between the exact solution and the approximate solutions for the 2nd sample path of $B$ ($d=1$)}
\end{figure}

\begin{figure}[H]
  \centering
  \includegraphics[width=0.5\textwidth]{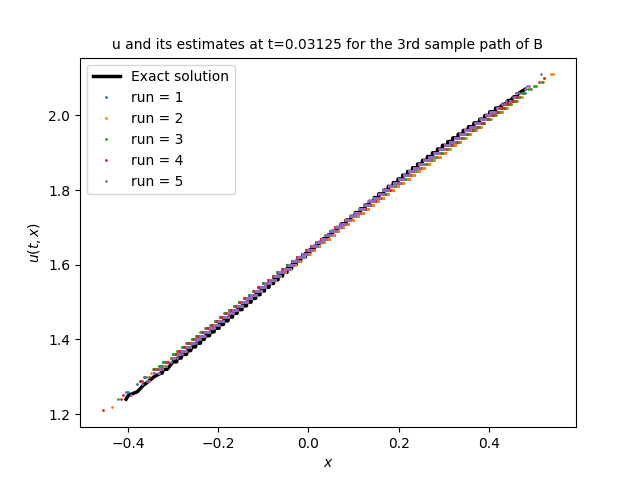}
  \caption{Comparison between the exact solution and the approximate solutions for the 3rd sample path of $B$ ($d=1$)}
\end{figure}

\begin{figure}[H]
  \centering
  \includegraphics[width=0.5\textwidth]{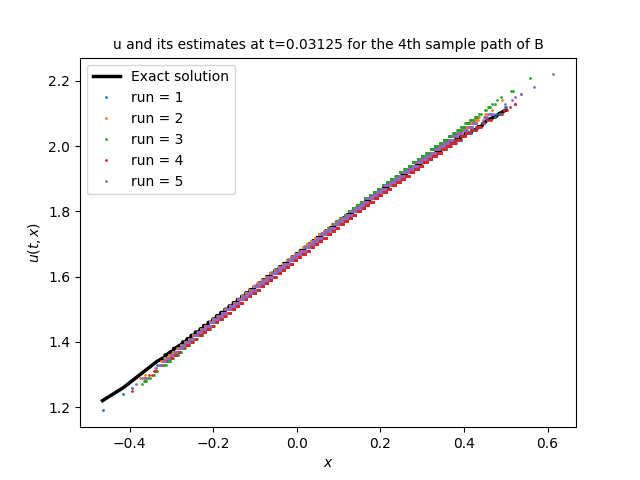}
  \caption{Comparison between the exact solution and the approximate solutions for the 4th sample path of $B$ ($d=1$)}
\end{figure}

\begin{figure}[H]
  \centering
  \includegraphics[width=0.5\textwidth]{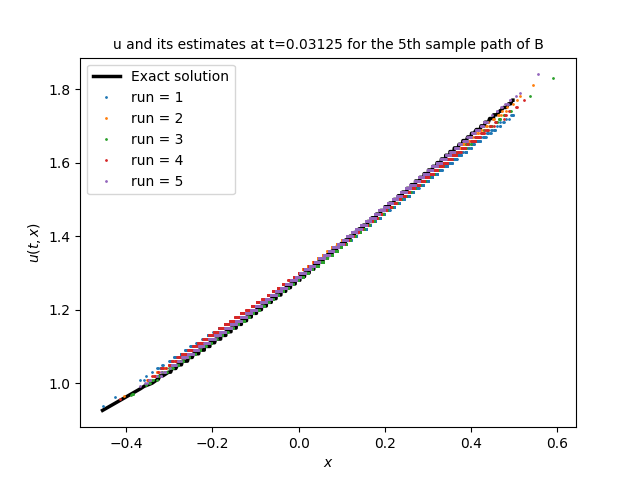}
  \caption{Comparison between the exact solution and the approximate solutions for the 5th sample path of $B$ ($d=1$)}
\end{figure}

Similarly, numerical results for the  5-dimensional, 50-dimensional and 100-dimensional cases are given in Tables 2-4, respectively. For these cases, the number of time steps is set to be 16. Each of the neural networks consists of 2 hidden layers, the number of nodes for each hidden layer is set to  be $d+10$ for 5-dimensional case, and $d+50$ for 50- and 100-dimensional cases. To illustrate the accuracy of the proposed algorithm, we also plot in Figures 6-10 the comparison between exact solution and the approximation solution of the 100-dimensional case as a function of $\frac{1}{100}\sum_{j=1}^{100} (X_t)_{j}$ at $t=0.0625$ for each realization of $B$ in Table 4.

\begin{table}[H]
    \centering
    \begin{tabular}{|l|l|l|l|}
    \hline
        Averaged Approx.  & Exact solution & Standard deviation & Relative error \\ \hline
        3.7816436 & 3.783391082 & 0.01756304 & 0.000461883 \\ \hline
        4.0550942 & 4.092778688 & 0.01245616 & 0.009207556 \\ \hline
        3.5838552 & 3.589354974 & 0.025385832 & 0.001532246 \\ \hline
        3.9931598 & 4.034977758 & 0.03498262 & 0.010363863 \\ \hline
        3.6655762 & 3.675363075 & 0.006242963 & 0.002662832 \\ \hline
        \multicolumn{4}{|l|}{Relative $L^2$ error:  0.00635359012791855 } \\ \hline
    \end{tabular}
\caption{Numerical results for the 5-dimensional case. }
\end{table}

\begin{table}[H]
    \centering
    \begin{tabular}{|l|l|l|l|}
    \hline
        Averaged Approx.  & Exact solution & Standard deviation & Relative error \\ \hline
        13.869227 & 13.80994108 & 0.0389653 & 0.004292989 \\ \hline
        11.789759 & 11.77779362 & 0.026603295 & 0.001015927 \\ \hline
        11.142225 & 11.12731684 & 0.02187582 & 0.00133978 \\ \hline
        8.79433 & 8.691944684 & 0.0351124 & 0.011779334 \\ \hline
        8.846188 & 8.819915756 & 0.02828248 & 0.002978741 \\ \hline
        \multicolumn{4}{|l|}{Relative $L^2$ error: 0.00581175481223784} \\ \hline
    \end{tabular}
\caption{Numerical results for the 50-dimensional case. }
\end{table}

\begin{table}[H]
    \centering
    \begin{tabular}{|l|l|l|l|}
    \hline
        Averaged Approx.  & Exact solution & Standard deviation & Relative error \\ \hline
        13.589519 & 13.50290922 & 0.06967104 & 0.006414157 \\ \hline
        10.588926 & 10.61844227 & 0.061020687 & 0.002779717 \\ \hline
        19.077723 & 18.69746132 & 0.048847217 & 0.02033761 \\ \hline
        14.551335 & 14.54826577 & 0.023958674 & 0.000210969 \\ \hline
        17.445202 & 17.53231125 & 0.022984497 & 0.004968498 \\ \hline
        \multicolumn{4}{|l|}{Relative $L^2$ error: 0.00987134339284379 } \\ \hline
    \end{tabular}
\caption{Numerical results for the 100-dimensional case. }
\end{table}

\begin{figure}[H]
  \centering
  \includegraphics[width=0.5\textwidth]{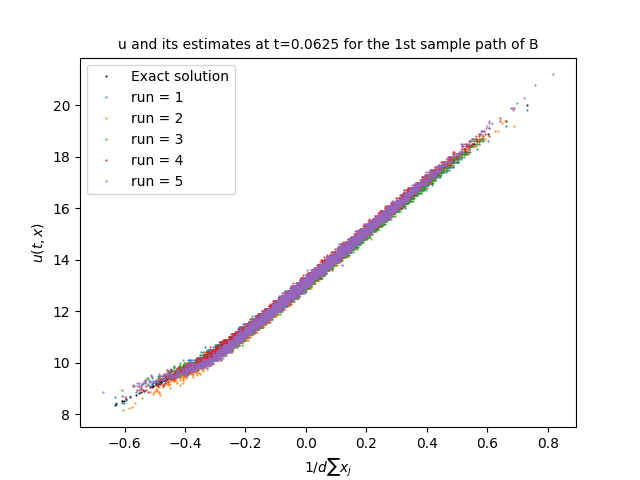}
  \caption{Comparison between the exact solution and the approximate solutions for the 1st sample path of $B$ ($d=100$)}
\end{figure}

\begin{figure}[H]
  \centering
  \includegraphics[width=0.5\textwidth]{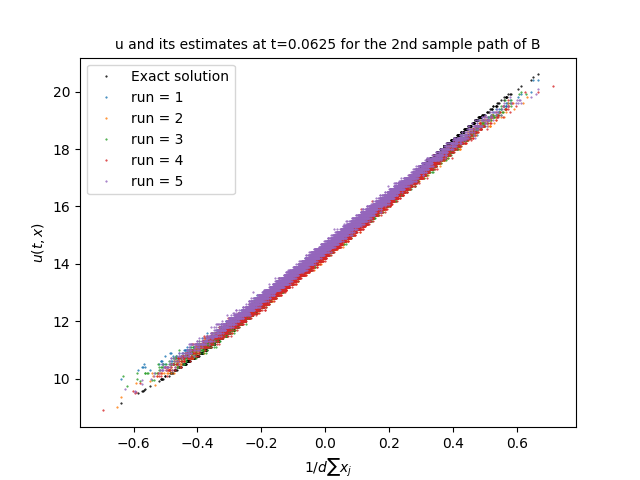}
  \caption{Comparison between the exact solution and the approximate solutions for the 2nd sample path of $B$ ($d=100$)}
\end{figure}

\begin{figure}[H]
  \centering
  \includegraphics[width=0.5\textwidth]{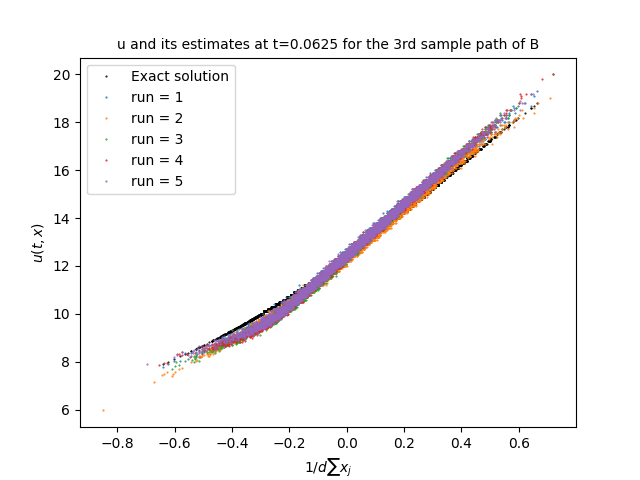}
  \caption{Comparison between the exact solution and the approximate solutions for the 3rd sample path of $B$ ($d=100$)}
\end{figure}
\begin{figure}[H]
  \centering
  \includegraphics[width=0.5\textwidth]{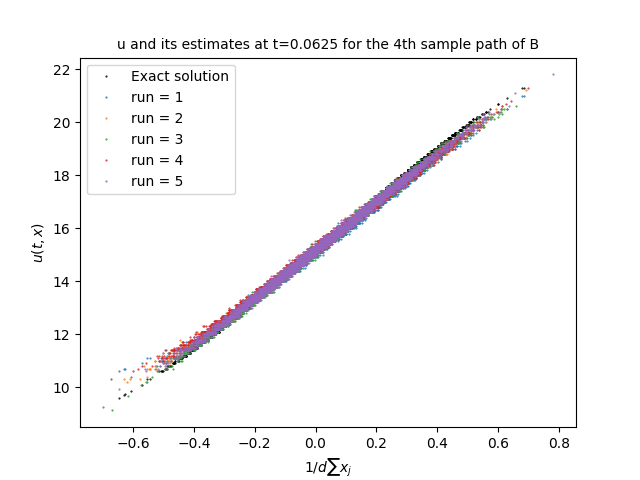}
  \caption{Comparison between the exact solution and the approximate solutions for the 4th sample path of $B$ ($d=100$)}
\end{figure}

\begin{figure}[H]
  \centering
  \includegraphics[width=0.5\textwidth]{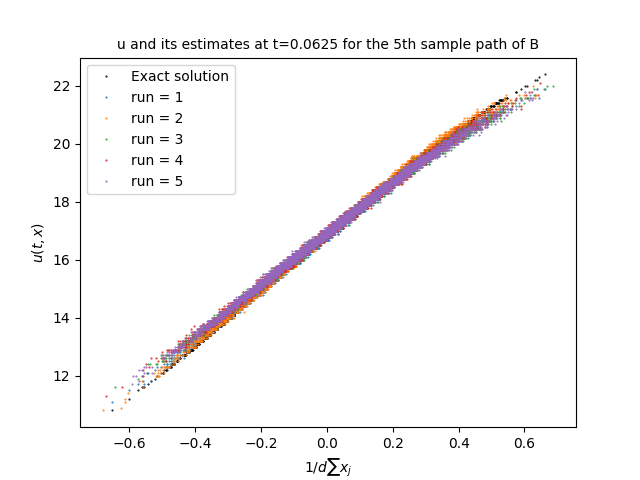}
  \caption{Comparison between the exact solution and the approximate solutions for the 5th sample path of $B$ ($d=100$)}
\end{figure}

\section{Conclusions}
We have proposed a predictor-corrector deep learning-based numerical method for solving high dimensional stochastic partial differential equations. At each time step, the original SPDE is first decomposed into a degenerate SPDE to serve as the prediction step, and a second-order deterministic PDE to serve as the correction step. The solution of the degenerate SPDE is then approximated by the Euler method, and the solution of the PDE is approximated by deep neural networks via the equivalent backward stochastic differential equation. The convergence analysis of the proposed algorithm is presented, and numerical examples are carried out to show the efficiency of the proposed algorithm. Future studies along this line include extending the current algorithm to solve nonlinear SPDEs and stochastic optimal control involving SPDEs.

\bibliography{deepspde}
\nocite{*}
\bibliographystyle{plain}

\end{document}